\documentclass{amsart}

\usepackage[T1]{fontenc}
\usepackage{amssymb}
\usepackage{amsaddr}

\newtheorem{theorem}[subsection]{Theorem}
\newtheorem{proposition}[subsection]{Proposition}
\newtheorem{corollary}[subsection]{Corollary}
\newtheorem*{claim}{Claim}

\theoremstyle{definition}
\newtheorem{definition}[subsection]{Definition}

\theoremstyle{remark}
\newtheorem*{remark}{Remark}

\newcommand{\set}[1]{\left\{ {#1} \right\}}
\newcommand{\hc}{\mathfrak{hc}}
\newcommand{\st}{\mathfrak{st}}
\newcommand{\ma}{\mathsf{MA}}
\newcommand{\continuum}{\mathfrak{c}}
\newcommand{\omegaone}{{\omega_1}}

\newcommand{\pred}{\operatorname{\mathrm{pred}}}

\title[Hadwiger conjecture and Characterizations of Trees Using Graphs]{The Uncountable Hadwiger Conjecture and Characterizations of Trees Using Graphs}

\author{Dávid Uhrik}
\address{Charles University, Faculty of Mathematics and Physics, Department of Algebra, Sokolovská 83, 186 75 Prague 8, Czech Republic}
\email{uhrik@karlin.mff.cuni.cz}
\address{Institute of Mathematics of the Czech Academy of Sciences, Žitná 25, 115 67 Prague 1, Czech Republic}

\thanks{This work has been supported by Charles University Research Center program No.UNCE/SCI/022.}

\begin{document}
\maketitle

\begin{abstract}
We prove that the existence of a non-special tree of size $\lambda$ is equivalent to the existence of an uncountably chromatic graph with no $K_\omegaone$ minor of size $\lambda$, establishing a connection between the special tree number and the uncountable Hadwiger conjecture. Also characterizations of Aronszajn, Kurepa and Suslin trees using graphs are deduced. A new generalized notion of connectedness for graphs is introduced using which we are able to characterize weakly compact cardinals.
\end{abstract}

\section{Introduction}
The Hadwiger conjecture is a deep unsolved problem in finite graph theory with far-reaching consequences. It states that if $G$ is a simple finite graph and the chromatic number of $G$ is $t$, then the complete graph on $t$ vertices is a minor of $G$. We are interested in generalizations of this conjecture to uncountable graphs.

Recently there has been some interest in the conjecture for infinite graphs. In \cite{zypen2013} van der Zypen proved that there is a countable connected graph whose chromatic number is $\omega$, but $K_\omega$ is not a minor of this graph, i.e.\ the straightforward generalization of the Hadwiger conjecture to infinite graphs fails. The proof can be generalized to limit cardinals. Later Komjáth \cite{komjath2017} showed that the Hadwiger conjecture fails for every infinite cardinal $\kappa$, he proved that:

\begin{theorem}[Komjáth \cite{komjath2017}] \label{hc_tree}
If $\kappa$ is an infinite cardinal, then there is a graph of cardinality $2^\kappa$, chromatic number $\kappa^+$, with no $K_{\kappa^+}$ minor.
\end{theorem}

Thus the Hadwiger conjecture fails also for uncountable graphs if there is no bound on the size of the witness. If we only consider graphs of size $\omega_1$, the conjecture may hold.

\begin{theorem}[Komjáth \cite{komjath2017}] \label{hc_ma}
If $\mathsf{MA}_{\omega_1}$ holds, then every graph $G$ with $|G| = \chi(G) = \omega_1$ contains a subdivision of $K_{\omega_1}$.
\end{theorem}

The proof of Theorem \ref{hc_tree} uses a non-$\kappa$-special tree of size $2^\kappa$ (such a tree always exists) to construct a graph with the desired properties. On the other hand, starting with a graph a tree can be obtained, using a specific construction by Brochet and Diestel \cite{brochetdiestel1994}, which reflects many useful properties of the graph. Theorems \ref{hc_tree} and \ref{hc_ma} give rise to a cardinal invariant.

\begin{definition}
The \emph{uncountable Hadwiger conjecture number}, denoted $\hc$, is the least size of an uncountably chromatic graph with no $K_{\omega_1}$ minor.
\end{definition}

In what follows we will show that $\hc$ is equal to the special tree number. For more information on the special tree number and its history see \cite{switzer2022}.

\begin{definition}
The \emph{special tree number}, denoted $\st$, is the least size of a non-special tree with no uncountable branch.
\end{definition}

What Komjáth essentially proved in \cite[Theorem 2]{komjath2017} for the case of $\kappa=\omega$ can be succinctly written as:

\begin{theorem}[Komjáth \cite{komjath2017}] \label{hc<st}
$\hc \le \st$.
\end{theorem}

In this note we prove the other inequality, thus establishing that in fact $\hc = \st$, see Corollary \ref{hc=st}.

Using the techniques in proving the previous theorem we also provide a characterization of Suslin trees in terms of graphs, proving that a Suslin tree exists if and only if there is a graph of size $\omegaone$ with no uncountable independent set and no $K_\omegaone$ minor, see Theorem \ref{suslin_characterization}. This characterization has been independently discovered by P. Komjáth and S. Shelah in their recent paper \cite{komjathshelah2021}.

In a similar fashion we define graph counterparts for Aronszajn and Kurepa trees using a generalized notion of connectedness for graphs. As a corollary we also deduce a new characterization of weakly compact cardinals using graphs.

\subsection*{Notation}
We use standard set theoretic notation. A graph $G$ is a pair $(V,E)$, where $V$ is an arbitrary set and $E \subseteq [V]^2$; the set $V$ is called the vertex set and $E$ the edge set of the graph. Unless otherwise specified $G$ will always denote a graph whose vertex set is $V$ and the edge set is $E$, in case of ambiguity we will use $V_G$ and $E_G$ instead. If $\kappa$ is a cardinal, then $K_\kappa$ denotes the complete graph on $\kappa$ vertices, i.e.\ the graph $(\kappa, [\kappa]^2)$. Two subsets $X,Y \subseteq V$ are connected if there are vertices $x \in X$ and $y \in Y$ such that $\set{x,y} \in E$. A graph is connected if between any two vertices there is a path. A component of a graph is any maximal connected subgraph. A graph is $\kappa$-connected if it stays connected after removing $<\kappa$ many vertices. If $X,Y,S$ are subsets of the vertex set we say that $S$ separates $X$ from $Y$ in $G$ if after removing $S$ from the graph, $X \setminus S$ and $Y \setminus S$ lie in different components.

The chromatic number of a graph $G$, denoted $\chi(G)$, is the least cardinal $\mu$ such that there is a proper coloring of $G$ with $\mu$ colors. A proper coloring is a function $c: V \to \mu$ such that $c(u) \neq c(v)$ whenever $\set{u,v} \in E$.

If $G$ and $H$ are graphs, we say that $G$ is a minor of $H$ if there are pairwise disjoint non-empty subsets $(X_u)_{u \in V_G}$ of $V_H$ such that each induces a connected subgraph in $H$ and for every $\set{u,v} \in E_G$, $X_u$ and $X_v$ are connected in $H$. Note also that being a minor is a transitive property. The graph $H$ is a subdivision of $G$ if $H$ is constructed from $G$ by replacing edges with paths. Note that if $H$ is a subdivision of $G$, then $G$ is a minor of $H$.

A tree is a poset $(T, \le)$ such that each set of the form $\set{s \in T \mid s < t}$ for some $t \in T$ is well-ordered, the order type of this set is called the height of the node $t$, denoted as $\mathrm{ht}(t)$. For a tree the set $T_\alpha$ is the set of nodes of height $\alpha$, sometimes referred to as level $\alpha$ of $T$, for $\alpha$ limit it will be useful to denote $T_{<\alpha}$, the set of nodes of height less than $\alpha$. By $\pred(t)$ we denote the set of nodes in the tree which are below $t$ in the tree order. By a branch in a tree we mean a maximal chain.

A tree is called $\kappa$-special if there is a function $f: T \to \kappa$ which is injective on chains. An $\omega$-special tree will be simply called a special tree.

A $\kappa$-Suslin tree is a tree of size $\kappa$ with no $\kappa$-branches and no antichains of size $\kappa$. A $\kappa$-Aronszajn tree is a tree of height $\kappa$ with no $\kappa$-branch and levels of size less than $\kappa$. A $\kappa$-Kurepa tree is a tree of height $\kappa$ with at least $\kappa^+$ different branches but levels of size less than $\kappa$. When $\kappa = \omegaone$ we omit the cardinal specification.

The comparability graph of a tree $T$ is a graph whose vertex set is the domain of the tree and the edge set is $\set{\set{s,t} \in [T]^2 \mid s \le t \lor t \le s}$.

If $T$ is a tree, a graph is called a $T$-graph if it is isomorphic to a graph whose vertex set is $T$, is a subgraph of the comparability graph of $T$ and for every $t \in T$ the vertices connected to $t$ are cofinal in $\set{s \in T \mid s < t}$, if this set has a maximum we want $t$ to be connected to it.

\subsection*{Acknowledgement.} The author is grateful to Chris Lambie-Hanson for constructive discussions on the topic which greatly improved the exposition of this paper.

\section{Limit cardinals}

First we show that the following result by van der Zypen can be easily generalized to cover every limit cardinal.

\begin{theorem}[van der Zypen \cite{zypen2013}]
There is a countable connected graph whose chromatic number is $\omega$ but $K_\omega$ is not a minor of this graph.
\end{theorem}

\begin{proposition} \label{hc_limit}
Suppose $\kappa$ is a limit cardinal. There exists a connected graph whose size and chromatic number is $\kappa$, but $K_\kappa$ is not a minor of this graph.
\end{proposition}
\begin{proof}
Suppose $(\kappa_\alpha \mid \alpha < \mu)$ is an increasing cofinal sequence of cardinals in $\kappa$, where $\mu$ may be equal to $\kappa$. Let the vertex set of the graph be the union $\{0\} \cup \set{(\alpha, \beta) \mid \alpha \in \mu \land \beta \in \kappa_\alpha}$. The edge set is defined as follows: connect the vertex $0$ to every other vertex and for $\alpha < \mu$ connect $(\alpha, \beta)$ to $(\alpha, \beta')$ for every $\beta < \beta' < \kappa_\alpha$. In simple terms the graph is the disjoint union of complete graphs of increasing size, cofinal in $\kappa$, plus a vertex connected to every other vertex.

The chromatic number of this graph is clearly bounded by $\kappa$ but also larger than $\kappa_\alpha$ for every $\alpha < \mu$ as the complete graph $K_{\kappa_\alpha}$ embeds into the graph. Thus the chromatic number is exactly $\kappa$.

Now assume $\set{C_\gamma \mid \gamma < \kappa}$ is a collection of pairwise disjoint connected subgraphs forming a $K_\kappa$ minor.

Suppose first that $0$ does not belong to any of the sets $C_\gamma$, then necessarily $\bigcup_{\gamma < \kappa} C_\gamma$ is a subset of some complete graph $K_{\kappa_\alpha}$, this is of course a contradiction as $\kappa_\alpha < \kappa$ and $C_\gamma$ contains at least one element.

So assume $0 \in C_{\gamma_0}$ for some $\gamma_0 < \kappa$. Now given any $\gamma \neq \gamma_0$ we must have that $C_\gamma \subseteq K_{\kappa_\alpha}$ for some $\kappa_\alpha$ otherwise they cannot be connected by an edge as $0$ is in $C_{\gamma_0}$ and it is the only vertex connecting the disjoint cliques. Hence we get a contradiction as in the previous case.
\end{proof}

Thus the Hadwiger conjecture fails unconditionally for $\omega$ and in general for every limit cardinal.

\section{On the $\hc$ number}

We prove some basic results about $\hc$ and show its connection to $\st$.

\begin{proposition}
\begin{enumerate}
    \item $\omega_1 \le \hc \le \st \le \continuum$,
    \item \label{ma_hc} $\mathsf{MA}$ implies $\hc = \continuum$,
    \item \label{sh_hc} for any $\kappa > \omega_1$ of uncountable cofinality it is consistent that $\omega_1 = \hc < \continuum = \kappa$.
\end{enumerate}
\end{proposition}
\begin{proof}
Clearly $\hc$ cannot be countable as the graph must be uncountably chromatic. The upper bound comes from Theorem \ref{hc<st}.

Assuming full $\ma$ in the proof of Theorem \ref{hc_ma} we obtain that in any uncountably chromatic graph of size $<\continuum$ a subdivision of $K_\omegaone$ can be found.

For (\ref{sh_hc}) consider the comparability graph of a Suslin tree. This defines a graph of size $\omegaone$ with no uncountable independent set as independent sets in comparability graphs clearly translate to antichains in the tree, hence the graph cannot be countably chromatic. The graph also has no $K_\omegaone$ minor as this would define an uncountable branch, see \cite[Theorem 2]{komjath2017} for a proof. As is well known \cite[Corollary V.4.14]{kunen2011} the continuum can be arbitrarily large while a Suslin tree exists.
\end{proof}

We can say something about the cofinality of $\hc$ as well.

\begin{proposition}
The cofinality of $\hc$ is uncountable.
\end{proposition}
\begin{proof}
Suppose $\kappa$ is a cardinal and $(\kappa_n \mid n \in \omega)$ is an increasing sequence of cardinals converging to $\kappa$. Assume that $\kappa_n < \hc$ for all $n \in \omega$. We will show that $\kappa$ is also strictly less than $\hc$.

Take a graph $G$ of size $\kappa$ whose chromatic number is $\omegaone$ and which has no $K_\omegaone$ minor. Partition the vertex set into countably many parts of increasing size corresponding to the cofinal sequence $(\kappa_n \mid n \in \omega)$ and consider the induced subgraphs $\{G_n \mid n \in \omega\}$. Now each graph has size less than $\hc$, thus it is either countably chromatic or contains $K_\omegaone$ as a minor. As the entire graph has no $K_\omegaone$ minor the latter is impossible, so it must be the case that $G_n$ is countably chromatic for every $n \in \omega$. Fix the countable colorings $c_n : G_n \to \omega$, we will show there is a countable coloring of the entire $G$. Define $c: G \to \omega\times\omega$ as $c(v) = (n,k)$ if $v \in G_n$ and $c_n(v) = k$. Now if $\{u,v\}$ is an edge in $G$ and $c(u) = c(v)$, then by the equality in the first coordinate both vertices are part of the same $G_n$ and by the equality in the second coordinate $c_n(u) = c_n(v)$. However, since these were induced subgraphs the edge $\{u,v\}$ is present in $G_n$ so $c_n$ is not a good coloring, a contradiction.
\end{proof}

Before proving the main theorem we need a result about $T$-graphs. The proof can be found in \cite[\textsection 2]{pitz2022}.

\begin{proposition} \label{branch_minor}
Suppose $\kappa$ is an uncountable regular cardinal. If $G$ is a $T$-graph and $T$ has a $\kappa$-branch, then $K_\kappa$ is a minor of $G$.
\end{proposition}

Our main result about the uncountable Hadwiger cojecture number will follow from the following more general theorem.

\begin{theorem} \label{special--chromatic}
Suppose $\kappa$ and $\lambda$ are infinite cardinals. The existence of a graph of size $\lambda$ whose chromatic number is $\kappa^+$ and has no $K_{\kappa^+}$ minor is equivalent to the existence of a non-$\kappa$-special tree of size $\lambda$ with no branches of length $\kappa^+$.
\end{theorem}

\begin{proof}
Given a non-$\kappa$-special tree of size $\lambda$ with no branches of length $\kappa^+$ the proof of \cite[Theorem 2]{komjath2017} shows that the comparability graph of this tree has chromatic number $\kappa^+$ and has no $K_{\kappa^+}$ minor, it is clear that the size of the comparability graph is the same as the size of the tree it is constructed from.

In the other direction assume $G$ is a graph with the aforementioned properties. We can assume that $G$ is connected as $G$ must have a component with the required properties and also that the vertex set of $G$ is $\lambda$. Note also that $\lambda \ge \kappa^+$. Using a construction by Brochet and Diestel, \cite[Theorem 4.2.]{brochetdiestel1994}, we will build a tree $T$ which will have the required properties. For completeness we reproduce the construction.

The construction is by induction. We will construct a tree level by level. To each node $t \in T$ we will associate two subsets $V_t \subseteq C_t$ of $V$ both inducing connected subgraphs. In the end the graph obtained by contracting the sets $V_t$ to a point will be a $T$-graph. For every $\alpha$ the following conditions will be satisfied:
\begin{enumerate}
    \item if $\alpha = \beta + 1$ and $t \in T_\alpha$ and $s$ is the predecessor of $t$ in $T_\beta$, then $C_t$ is a component of $C_s \setminus V_s$ and $V_t = \set{x_t}$ such that $x_t \in C_t$ and $x_t$ is connected to some element in $V_s$; also if $C$ is a component of $C_s \setminus V_s$, then there is a unique node $r \in T_\alpha$ above $s$, such that $C = C_r$,
    \item if $\alpha$ is limit and $t \in T_\alpha$, then $C_t$ is a component of $\bigcap_{s \in \pred(t)} C_s$ and $V_t$ induces a non-empty connected subgraph of $C_t$ such that the following set $\set{s \in \pred(t) \mid V_s \text{ is connected to } V_t \text{ in } G}$ is cofinal in $\pred(t)$; also if $C$ is a component of $\bigcap_{s \in \pred(t)} C_s$, then there is a unique node $r \in T_\alpha$ above $\pred(t)$, such that $C = C_r$,
    \item if $t \in T_{\alpha}$, then $s \in \pred(t)$ if and only if $C_t \subsetneq C_s$.
\end{enumerate}

The first step is to define the root $T_0 := \{t_0\}$ and $V_{t_0} := \set{0}$, where $0 \in \lambda = V$ and $C_{t_0} := V$. We now continue the construction by levels.

If we are at a successor stage, i.e.\ we have defined the tree up to level $\alpha + 1$ we do the following: for every node $t$ in $T_\alpha$ consider the graph induced by $C_t \setminus V_t$ and let $\set{C_i \mid i < \mu(t)}$ be all of its components. In $T_{\alpha+1}$ the node $t$ will have $\mu(t)$ many successors $\set{t^i \mid i < \mu(t)}$ and we put $C_{t^i} := C_i$. Let $y$ be the least element of $C_{t^i}$ (we assumed $V = \lambda$) and put $V_{t^i} := \set{x^i_t}$, where $x^i_t \in C_{t^i}$ is a vertex connected to $V_t$ whose distance to $y$ is minimal. Note that if $C_t \setminus V_t$ is empty, $t$ will have no successors in $T$.

At limit stages we proceed similarly. Suppose that we have defined the tree up to level $\alpha$ and $\alpha$ is a limit ordinal. We consider every branch $b$ of $T_{<\alpha}$ and determine the set $\bigcap_{s \in b} C_s$, if it is empty $b$ will be a branch in $T$ as well, otherwise let $\set{C_i \mid i < \mu(b)}$ be all of its components. We put nodes $(b^i)_{i < \mu(b)}$ into $T_\alpha$, all of them extending $b$ and pairwise incomparable. Fix a $j < \mu(b)$. The set $C_{b^j}$ is simply $C_j$ again. It is enough now to define the set $V_{b^j}$.

\begin{claim}
There is a subset $V'$ of $C_j$ inducing a connected subgraph in $G$ of size at most $\mathrm{cf}|\alpha|$ such that the set of nodes $\set{t \in b| V_t \text{ is connected to } V' \text{ in }G}$ is cofinal in $b$.
\end{claim}
\begin{proof}
First we show that $\set{t \in b| V_t \text{ is connected to } C_j \text{ in }G}$ is cofinal in $b$. Suppose $s$ is a node in $b$. Note that $C_j \subseteq C_s$ so there must be vertices $u \in C_j$ and $v \in C_s \setminus C_j$ which are connected. As $v \not\in C_j$ let $s' \in b$ be the least node such that $v \not\in C_{s'}$, note that we have $u \in C_{s'}$.

First suppose $\mathrm{ht}(s')$ is a successor ordinal, and the predecessor of $s'$ is $\overline{s}$. We obtain that $u,v \in C_{\overline{s}}$ but $v \not\in C_{s'}$ and since $u$ is connected to $v$ in $G$ we must have that $v \in V_{\overline{s}}$. Hence $V_{\overline{s}}$ is connected to $C_j$.

The case when the height of $s'$ is limit we obtain that $u,v \in \bigcap_{r \in \pred(s')} C_r$. However since $\{u,v\}$ forms an edge both vertices must be contained in the same component of $\bigcap_{r \in \pred(s')} C_r$, i.e.\ $v$ would have to be an element of $C_{s'}$.

Let $b' \subseteq b$ be a cofinal subset of the branch of size $\mathrm{cf}|\alpha|$ such that each $s \in b'$ has the property that $V_s$ is connected to $C_j$. For each $s \in b'$ choose a witness $x^i_s \in C_j$ connected to $V_s$. Let $V'$ be any subset of $C_i$ inducing a connected graph of size at most $\mathrm{cf}|\alpha|$ containing all the vertices $x^i_s$.
\end{proof}

The claim defines the sets $V_{b^i}$ and the construction is finished.

It is clear that in each step of the induction we use up at least one vertex of $G$ that we put in some $V_t$, so the length of the induction is some ordinal $\delta$ such that $|\delta| = \lambda$. However note that at successor steps of the induction we always chose a vertex with minimal distance from the least vertex which was not part of some $V_t$ defined before. This implies that the vertex $\alpha < \lambda$ was the least vertex at the latest in step $\omega \cdot \alpha$ of the induction and was put into some $V_t$ after finitely many steps, i.e.\ $\alpha$ belongs to some $V_t$ for a $t$ such that $\mathrm{ht}(t) < \omega\cdot\alpha + \omega$. From this we also obtain that the height of the constructed tree is at most $\lambda$. 

We make a few observations. The sets $(V_t)_{t \in T}$ form a partition of $V$ and each $V_t$ induces a connected subgraph of $G$. The graph $(T, F)$, where $$\set{s, t} \in F \mskip6mu \equiv \mskip6mu V_s \text{ is connected to } V_t \text{ in } G$$ is a $T$-graph and also a minor of $G$.

To see that it is a $T$-graph, note that the way we defined the sets $V_t$ it is clear that each $t$ is $F$-related to its predecessor and if the height of $t$ is limit the previous claim implies that it is cofinally often connected to its predecessors.

Thus it is enough to see that $(T, F)$ is a subgraph of the comparability graph of $T$. Consider any $u, v \in G$ which are connected and the corresponding $V_s$ and $V_t$ so that $u \in V_s, v \in V_t$. Suppose $s$ is incomparable with $t$.

In the first case assume that there is a limit $\alpha$ and a branch $b$ in $T_{<\alpha}$ such that there are nodes $\overline{s}$ and $\overline{t}$ directly above $b$ such that $\overline{s} \le s$ and $\overline{t} \le t$. We have that $u,v \in \bigcap_{r \in b} C_r$ and since they form an edge both $u$ and $v$ belong to the same component of $\bigcap_{r \in b} C_r$ and by the second induction hypothesis we obtain that $\overline{s} = \overline{t}$.

The case when the split happens on a successor level is proven analogously using the first induction hypothesis.

The following claim finishes the proof.

\begin{claim}
$T$ is a non-$\kappa$-special tree without $\kappa^+$-branches of size at most $\lambda$.
\end{claim}
\begin{proof}\renewcommand{\qedsymbol}{}
The size of $T$ is obviously at most $\lambda$. Suppose $T$ has a $\kappa^+$-branch, then by Proposition \ref{branch_minor} $K_{\kappa^+}$ is a minor of $(T,F)$ thus by transitivity it is also a minor of $G$ which is impossible. Thus all branches have size $\le \kappa$ and so each $V_t$ has size $\le \kappa$.

If $T$ were $\kappa$-special there would be a specializing function $f: T \to \kappa$. Let $\set{A_\alpha \mid \alpha \in \kappa}$ be pairwise disjoint $\kappa$-sized subsets of $\kappa$. For every $t \in T$ let $g_t: V_t \to A_{f(t)}$ be any injection and define a coloring $c: G \to \kappa$ as follows: for every $u \in G$ there is a unique $t \in T$ such that $u \in V_t$ thus let $c(u)$ be $g_t(u)$. To see that $c$ is proper, consider any two vertices $u,v \in G$ which are connected. If both are in the same $V_t$ then as $g_t$ is injective they clearly get different colors, else there are different $s, t \in T$ so that $u \in V_s$ and $v \in V_t$, since $(T,F)$ is a $T$-graph we get that $s < t$ or vice versa, then however, $f(s) \neq f(t)$ and subsequently $c(u) \in A_{f(s)}$ and $c(v) \in A_{f(t)}$ and again they get different colors. Thus the chromatic number of $G$ is at most $\kappa$, a contradiction.
\end{proof}
\end{proof}
\begin{remark}
From now on if a graph $G$ is given we will freely denote by $T_G$ the tree arising from the previous construction without explicitly mentioning so.
\end{remark}

Using this theorem we can easily deduce the relationship between $\hc$ and $\st$, thus the inequality in Theorem \ref{hc<st} can be reversed.

\begin{corollary} \label{hc=st}
$\hc = \st$.
\end{corollary}
\begin{proof}
We show that $\hc \ge \st$. Given a graph of size $\hc$ with uncountable chromatic number and no $K_{\omegaone}$ minor the previous theorem implies that there is a non-special tree with no uncountable branch of size at most $\hc$.
\end{proof}

The previous construction has many useful properties. Suppose $\kappa$ is a regular cardinal. We have the following:
\begin{enumerate}
    \item if $T_G$ has a $\kappa$-branch, then $K_\kappa$ is a minor of $G$,
    \item if $T_G$ is $\kappa$-special, then the chromatic number of $G$ is at most $\kappa$,
    \item if $T_G$ has a $\kappa$-sized antichain, then $G$ has an independent set of size $\kappa$.
\end{enumerate}
The first two items follow from Proposition \ref{branch_minor} and the last claim in the main theorem. The last item will follow from the characterization in the next section.

We can further show that the implication in item (1) can be reversed but not in (2) nor (3).

\begin{proposition} \label{t_g_properties}
Suppose $\kappa$ is a regular cardinal. Let $G$ be a graph, Then the following holds:
\begin{enumerate}
    \item if $K_\kappa$ is a minor of $G$, then $T_G$ has a branch of size at least $\kappa$,
    \item there is a countably chromatic graph $G$ such that $T_G$ has a $\kappa$-branch,
    \item there is a graph with an independent set of size $\kappa$ such that $T_G$ is isomorphic to $(\kappa, \in)$.
\end{enumerate}
\end{proposition}
\begin{proof}
For (1) suppose $\set{U_\alpha \mid \alpha < \kappa}$ are the disjoint sets of vertices forming a $K_\kappa$ minor and assume also that there is no node in $T_G$ at height $\kappa$, otherwise there clearly is a branch of length at least $\kappa$. For each $\alpha$ we define a node in the tree $T_G$, put $t_\alpha := \min \set{t \in T_G \mid U_\alpha \cap V_t \neq \emptyset}$ (the sets $V_t$ are defined as in Theorem \ref{special--chromatic}).
\begin{claim}
For each $\alpha$ the node $t_\alpha$ is well-defined.
\end{claim}
\begin{proof}
Suppose $s,t$ are incomparable such that $U_\alpha \cap V_s \neq \emptyset$ and also $U_\alpha \cap V_t \neq \emptyset$. Let $x_s$ be an element of $U_\alpha \cap V_s$ and $x_t$ an element of $U_\alpha \cap V_t$. As $U_\alpha$ is connected, there is a finite path $(x_i)_{i < n}$ in $U_\alpha$ such that $x_0 = x_s$ and $x_{n-1} = x_t$. To each $x_i$ we van associate a node $r_i$ such that $x_i \in V_{r_i}$ (this mapping need not be injective). Since $x_i$ is connected to $x_{i+1}$ we obtain that $r_i$ is comparable with $r_{i+1}$ for each $i < n-1$. Now there must exist some $r_j$ such that $r_j \le r_i$ for each $i \neq j$. We can prove this by induction. If $n\le3$ this is clear. Suppose $k < n$ and $(r_i)_{i<k}$ has a least element, say $r_l$. Now $r_l \le r_{k-1}$ and $r_k$ is comparable with $r_{k-1}$, hence $r_k$ is comparable with $r_l$ and the least element of $(r_i)_{i \le k}$ is $\min \set{r_k, r_l}$.

Let $r_j$ be the least element of $(r_i)_{i<n}$, then $r_j \le s, t$ and $x_j \in V_{r_j} \cap U_\alpha$.
\end{proof}

We claim that for every $\alpha, \beta < \kappa$ the nodes $t_\alpha$ and $t_\beta$ are comparable. If $\alpha$ and $\beta$ are given consider the nodes $s_\alpha, s_\beta \in T_G$ such that $x \in U_\alpha$ and $y \in U_\beta$ are connected in $G$ and $x \in V_{s_\alpha}$ and $y \in V_{s_\beta}$, so $s_\alpha$ is comparable with $s_\beta$. We then have that $t_\alpha, t_\beta \le \max \set{s_\alpha, s_\beta}$, hence $t_\alpha$ and $t_\beta$ are comparable. We are almost done but notice that the mapping $\alpha \mapsto t_\alpha$ need not be injective, however, as the sets $\set{U_\alpha \mid \alpha < \kappa}$ are pairwise disjoint and the sets $V_t$ have size less than $\kappa$, this mapping has the property that the preimage of each node has size $<\kappa$ and so there must be $\kappa$ many unique nodes $t_\alpha$ all comparable to each other forming a branch in $T_G$.

Now (2) is easy, just consider the complete graph on $\kappa$ vertices and subdivide each edge once, this is clearly a bipartite graph, hence $2$-colorable but from the previous part $T_G$ will have a $\kappa$-branch.

The witness for (3) is simply the complete bipartite graph with both partitions of size $\kappa$, clearly, either partition forms an independent set of size $\kappa$. Note that this graph is $\kappa$-connected. Following the construction in Theorem \ref{special--chromatic} it is not hard to see that $\kappa$-connectivity implies that at no step of the construction does the graph split, i.e.\ the tree $T_G$ is simply a single $\kappa$-branch.
\end{proof}

Let us conclude this section with an observation about the connectedness of the graphs which are counterexamples to the uncountable Hadwiger conjecture. The next proposition claims that a graph with no $K_\kappa$ minor cannot be $\kappa$-connected.

\begin{proposition} \label{minor_connectedness}
Suppose $\kappa$ is an infinite cardinal. If $G$ is $\kappa$-connected, then $G$ contains a subdivision of $K_{\kappa}$.
\end{proposition}
\begin{proof}
We will inductively choose elements $\set{v_\alpha \mid \alpha < \kappa} \subseteq V$ and finite sequences of vertices $\set{\mathbf{p}_{\alpha\beta} \mid \alpha < \beta < \kappa}$ so that for each pair of vertices $v_\alpha, v_\beta$ the sequence ${v_\alpha}^\smallfrown {\mathbf{p}_{\alpha \beta}}^\smallfrown v_\beta$ is a path and the collection of all these $\mathbf{p}_{\alpha\beta}$'s is pairwise disjoint. Clearly this forms a subdivision of $K_\kappa$.

Suppose we have constructed $\set{v_\alpha \mid \alpha < \gamma}$ for some $\gamma < \kappa$ and we also have the collection of pairwise disjoint paths $\set{\mathbf{p}_{\alpha\beta} \mid \alpha < \beta < \gamma}$. Consider the subgraph of $G$ induced by the vertices $V_\gamma := V \setminus \bigcup \set{\mathbf{p}_{\alpha\beta} \mid \alpha < \beta < \gamma}$. By our assumption this still induces a connected graph, choose any vertex $v$ from this set different from any vertex included in $\set{v_\alpha \mid \alpha < \gamma}$, this will be the vertex $v_\gamma$.

By induction again we choose the finite sequences $\mathbf{p}_{\alpha\gamma}$. Suppose we have constructed $\set{\mathbf{p}_{\beta\gamma} \mid \beta < \alpha}$ for some $\alpha < \gamma$.  As paths are finite we have that the set $V_\gamma \setminus \bigcup \set{\mathbf{p}_{\beta\gamma} \mid \beta < \alpha}$ still induces a connected graph as we assume $G$ is $\kappa$-connected, so choose any path between $v_\alpha$ and $v_\gamma$ using only these vertices excluding the so far chosen $\set{v_\alpha \mid \alpha < \gamma}$ and this will be our $\mathbf{p}_{\alpha\gamma}$.
\end{proof}

\section{Suslin trees}

Before we state the characterization of Suslin trees we mention a related result by Wagon. He introduced a graph counterpart to a Suslin tree whose existence is actually equivalent to the existence of a Suslin tree, in \cite[Theorem 2.2.]{wagon1978} he proved the following:

\begin{theorem}[Wagon \cite{wagon1978}]
For an infinite cardinal $\kappa$, the following are equivalent:
\begin{enumerate}
    \item a $\kappa^+$-Suslin tree exists,
    \item there exists a triangulated graph $G$ with $\alpha(G) = \kappa < \theta(G)$.
\end{enumerate}
\end{theorem}
The number $\alpha(G)$ denotes the supremum of the sizes of independent sets of $G$ and $\theta(G)$ the least size of a family of cliques which cover the graph. A graph is triangulated if every induced cycle is a triangle.

In our characterization we drop the triangularity requirement and instead of considering the invariant $\theta(G)$ we forbid a specific minor. This result has been independently proven by P. Komjáth and S. Shelah \cite{komjathshelah2021}.

\begin{theorem} \label{suslin_characterization}
Suppose $\kappa$ is a regular cardinal. The existence of a $\kappa$-Suslin tree is equivalent to the existence of a graph of size $\kappa$, which has no independent set of size $\kappa$ and has no $K_{\kappa}$ minor.
\end{theorem}
\begin{proof}
Given a $\kappa$-Suslin tree the comparability graph has the desired properties. It clearly has no independent set of size $\kappa$ as this would translate to a $\kappa$ sized antichain in the tree and a $K_{\kappa}$ minor translates to a $\kappa$-branch, for details see \cite[Theorem 2]{komjath2017}.

In the other direction we proceed as before. Given a graph $G$ of size $\kappa$ with no independent set of size $\kappa$ and no $K_{\kappa}$ minor we construct the tree $T_G$, the size of the tree is $\kappa$. Clearly the fact that the graph has no $K_{\kappa}$ minor again translates to the fact that the tree has no $\kappa$-branch.

To see that $T_G$ has no antichain of size $\kappa$ we proceed by contradiction. Suppose $\set{t_i \mid i \in \kappa}$ is an antichain in $T_G$ and consider the sets $\set{V_{t_i} \mid i \in \kappa}$. We showed that if $s$ is incomparable with $t$ then there is no edge connecting the sets $V_s$ and $V_t$, thus choose any $x_i \in V_{t_i}$ as all of these are non-empty, now it is clear that $\set{x_i \mid i \in \kappa}$ forms an independent set in $G$, a contradiction.
\end{proof}

\section{Narrow and Kurepa trees}

In this section we will try to characterize Aronszajn and Kurepa trees similarly to what has been done in the previous section. First we introduce a generalized notion of connectedness for graphs.

\begin{definition}
Suppose $\kappa, \lambda$ are infinite cardinals. A graph is $(\kappa,\lambda)$-connected if after the removal of less than $\kappa$ vertices the number of components is non-zero and less than $\lambda$.
\end{definition}
\begin{remark}
Note that the proof of Proposition \ref{hc_limit} shows that if $\kappa$ is singular, then the constructed graphs are $(\kappa, \kappa)$-connected; this is not the case for regular $\kappa$.
\end{remark}

Evidently, a graph is $\kappa$-connected when it is $(\kappa, 2)$-connected. This notion aims to stratify the property of connectedness for graphs. Typical examples of $(\kappa, \lambda)$-connected graphs come from trees and this notion is closely related to their width, e.g.\ the comparability graph of an Aronszajn tree is $(\omegaone, \omegaone)$-connected.

\begin{proposition}
Suppose $\kappa$ is an infinite cardinal. If $G$ is a graph of size at least $\kappa$ and $G$ has no independent set of size $\kappa$, then $G$ is $(\kappa, \kappa)$-connected.
\end{proposition}
\begin{proof}
Suppose $G$ is not $(\kappa, \kappa)$-connected, then there exists a set $X \subseteq V$ of size less than $\kappa$ such that if $X$ is removed from $G$, the graph splits into at least $\kappa$ many components, $\{C_\gamma \mid \gamma < \kappa\}$. From each component choose a vertex, $x_\gamma \in C_\gamma$. Now $\set{x_\gamma \mid \gamma < \kappa}$ forms an independent set in $G$.
\end{proof}

\begin{proposition} \label{t_g_narrow}
Suppose $\kappa \ge \lambda$ are regular cardinals. If $G$ is $(\kappa, \lambda)$-connected graph, then $(T_G)_{<\kappa}$ has levels of size $<\lambda$.
\end{proposition}
\begin{proof}
We will proceed by induction, clearly by our assumption the set of roots of $T_G$ has size $<\lambda$. Let $\alpha < \kappa$ and consider the level $(T_G)_{\alpha}$. Take the union $\bigcup_{t \in (T_G)_{<\alpha}} V_t$ (for definition of $V_t$ see Theorem \ref{special--chromatic}) and note that by construction of $T_G$ and the induction hypothesis this set has size $<\kappa$ so removing these vertices from $G$ leaves us with less than $\lambda$ many components, so at stage $\alpha$ in the construction of $T_G$ there are less than $\lambda$ many components to consider and hence less than $\lambda$ many nodes to extend $(T_G)_{<\alpha}$.
\end{proof}

\begin{corollary}
Suppose $\kappa$ is a regular cardinal. If there exists a cardinal $\lambda < \kappa$ such that $G$ is $(\kappa, \lambda)$-connected, then $K_\kappa$ is a minor of $G$.
\end{corollary}
\begin{proof}
Given $G$ with these properties consider the tree $(T_G)_{<\kappa}$. The size of the levels of this tree is less than $\lambda$ but the size of the entire tree is $\kappa$. By a result of Kurepa \cite{kurepa1935} each tree of height $\kappa$ whose levels have size less than $\lambda$ has a cofinal branch, hence $T_G$ has a branch of size at least $\kappa$ so by Proposition \ref{branch_minor} $G$ has a $K_\kappa$ minor.
\end{proof}

Considering the basic properties of $T_G$ together with the last proposition we have the following.

\begin{theorem}
Suppose $\kappa$ is a regular cardinal. The existence of a $\kappa$-Aronszajn tree is equivalent to the existence of a $(\kappa, \kappa)$-connected graph of size $\kappa$, which has no $K_{\kappa}$ minor.
\end{theorem}
\begin{proof}
Let $T$ be a $\kappa$-Aronszajn tree. We show that the comparability graph $G_T$ of this tree has the desired properties. Note that removing less than $\kappa$ many vertices from $G_T$ we can assume all of them lie below some level $\alpha$ of $T$ and hence if they are deleted we end up with less than $\kappa$ many connected subgraphs as the levels of $T$ have size less than $\kappa$. For the second property see the proof of \cite[Theorem 2]{komjath2017} using the fact that $T$ has no $\kappa$-branch. We can also prove the converse. If $G_T$ had the necessary properties then $T$ would be $\kappa$-Aronszajn to begin with. Clearly, a $\kappa$-branch in $T$ would imply the complete graph $K_\kappa$ is a subgraph of $G_T$ and if $T$ had levels of size $\kappa$ consider the first level $\alpha < \kappa$ which has $\kappa$ many nodes, then removing those vertices from $G_T$ which correspond to nodes in $T$ of height less than $\alpha$ (there are less than $\kappa$ many of those) we would end up with $\kappa$ many disjoint connected components of $G_T$ which is a contradiction as we assumed it is $(\kappa, \kappa)$-connected.

If $G$ is $(\kappa, \kappa)$-connected of size $\kappa$ with no $K_{\kappa}$ minor, then by the previous proposition $T_G$ has levels of size less than $\kappa$ and by the properties of the construction of $T_G$, it can have no $\kappa$-branch as this would imply $K_\kappa$ is a minor of $G$; from this we also have that $T_G$ has height $\kappa$ as $G$ must have size $\kappa$ as it is $(\kappa, \kappa)$-connected, hence $T_G$ has size $\kappa$. We can prove the converse here as well, $T_G$ being $\kappa$-Aronszajn implies that $G$ is $(\kappa, \kappa)$-connected and has no $K_\kappa$ minor. By Proposition \ref{t_g_properties} we get that since $T_G$ has no cofinal branch, then $G$ cannot have a $K_\kappa$ minor, also if less than $\kappa$ many vertices are removed from $G$, then there is a level $\alpha < \kappa$ such that all of these vertices are contained in the sets $V_t$ for $t \in (T_G)_{<\alpha}$, however, as $T_G$ has levels of size less than $\kappa$ then removing all of these vertices leaves us with less than $\kappa$ many cones in $T_G$ all of which define connected subgraphs of $G$.
\end{proof}

\begin{corollary}
Suppose $\kappa$ is an inaccessible cardinal. The following are equivalent:
\begin{enumerate}
    \item $\kappa$ is weakly compact,
    \item each $(\kappa, \kappa)$-connected graph has a $K_{\kappa}$ minor.
\end{enumerate}
\end{corollary}

For a more succinct characterization of Kurepa trees we define a so called \emph{Kurepa minor family} in a graph.

\begin{definition}
Suppose $\kappa$ and $\lambda$ are infinite cardinals. Let $G$ be a graph and $\set{W_\alpha \mid \alpha < \lambda}$ a collection of $K_\kappa$ minors of $G$. We say that $\set{W_\alpha \mid \alpha < \lambda}$ forms a \emph{$\kappa$-Kurepa minor family} of size $\lambda$ if for each $\alpha$ and $\beta$ a set of size less than $\kappa$ separates $W_\alpha$ from $W_\beta$ in $G$.
\end{definition}

\begin{theorem}
Suppose $\kappa$ is a regular cardinal. The existence of a $\kappa$-Kurepa tree is equivalent to the existence of a $(\kappa, \kappa)$-connected graph of size $\kappa$, which has a $\kappa$-Kurepa minor family of size at least $\kappa^+$.
\end{theorem}
\begin{proof}
If $T$ is a $\kappa$-Kurepa tree then the comparability graph of $T$ clearly has all the necessary properties.

On the other hand given such a graph the tree $T_G$ is a $\kappa$-Kurepa tree. From the construction we obtain that the height of $T_G$ is at most $\kappa$. The size of the levels is less than $\kappa$ by Proposition \ref{t_g_narrow}.

As for the branches, we get from Proposition \ref{t_g_properties} that each $K_\kappa$ minor defines a $\kappa$-branch in $T_G$. It is enough to observe that since $G$ has a $\kappa$-Kurepa minor family of size $\kappa^+$ we get that each pair of branches coming from this family is eventually different. Since each pair of minors is separated by a set, $X$, of size less than $\kappa$ we obtain that there is some $\alpha < \kappa$ such that the set of nodes $t$ with the property that $V_t$ intersects $X$ lie in $(T_G)_{<\alpha}$, this implies that the branches defined from the minors are indeed different.
\end{proof}

\section{The $\kappa$-Hadwiger conjecture} \label{sec_generalization}

Theorem \ref{special--chromatic} was used to prove that $\hc = \st$ but it also shows us where to look for models of the Hadwiger conjecture on higher cardinals.

\begin{definition}
The $\kappa$-Hadwiger conjecture states that every graph of size $\kappa$ whose chromatic number is $\kappa$ has a $K_\kappa$ minor.
\end{definition}

The case when $\kappa=\omegaone$ consistently holds, see Theorem \ref{hc_ma}, and by our result is equivalent to $\st > \omegaone$. For $\kappa$ limit the conjecture always fails as shown in Proposition \ref{hc_limit}. The generalized continuum hypothesis implies that the $\kappa$-Hadwiger conjecture fails for every infinite $\kappa$, this follows from Theorem \ref{hc_tree}.

Using the technique of Laver and Shelah (see the closing remarks in \cite{lavershelah1973}) we get a model where each tree with no $\kappa^+$-branch of size $\kappa^+$ is $\kappa$-special, except possibly at successors of singular cardinals. By Theorem \ref{special--chromatic} this also models the $\kappa^+$-Hadwiger conjecture.

\begin{theorem}
Suppose $\kappa$ is a regular cardinal. It is consistent that the $\kappa^+$-Hadwiger conjecture holds.
\end{theorem}

\section{Closing remarks}
As far as the author is concerned the construction of Brochet and Diestel \cite[Theorem 4.2.]{brochetdiestel1994} is not very well known in the set theory community and the question is how far can we push graph properties onto trees and vice versa. We have seen that chromaticity, having a large minor, independent sets and notions of connectedness all translate to properties of the tree. On the other hand the simple construction of a comparability graph from a tree reveals some graph properties arising from trees.

\end{document}